\documentclass[11pt]{amsart}

\usepackage{amsthm,amsmath,amssymb,natbib,setspace}

\newcommand{\tv}{\mathbf{t}}
\newcommand{\Xv}{\mathbf{X}}
\newcommand{\xv}{\mathbf{x}}
\newcommand{\Yv}{\mathbf{Y}}
\newcommand{\yv}{\mathbf{y}}
\newcommand{\Am}{\mathbf{A}}
\newcommand{\Bm}{\mathbf{B}}
\newcommand{\Sig}{\mathbf{\Sigma}}
\newcommand{\mv}{\mathbf{m}}
\newcommand{\Sym}{\mathfrak{S}}
\newcommand{\bLambda}{\mathbf{\Lambda}}

\newcommand{\threebythree}[9]{\left[
    \begin{array}{ccc} #1 &#2 &#3 \\ #4 &#5 &#6\\ #7 &#8 &#9 \end{array}
\right]}     

\newcommand{\CovFive}
{\left[
    \begin{array}{ccccc} 
    1 & 0 & 0 & 0 & -1/5 \\
    0 & 1/2 & 0 & -1/8 & -1/10 \\
    0 & 0 & 2/9 & -1/12 & -1/15\\
    0 & -1/8 & -1/12 & 3/16 & -1/20 \\
    -1/5 & -1/10 & -1/15 & -1/20 & 4/25
  \end{array} \right]
  }
  
 \newtheorem{theorem}{Theorem}

\newtheorem{lemma}[theorem]{Lemma}
\newtheorem{corollary}[theorem]{Corollary}

\theoremstyle{remark}
\newtheorem{remark}[theorem]{Remark}

\theoremstyle{definition}
\newtheorem{example}[theorem]{Example}

\numberwithin{equation}{section}
\numberwithin{theorem}{section} 
  
\begin{document}

\title{Integer Partitions Probability Distributions}
 \author{Andrew V. Sills}

\maketitle

\begin{abstract}
Two closely related discrete probability distributions are introduced.
In each case the support is a set of vectors in $\mathbb{R}^n$ obtained from the
partitions of the fixed positive integer $n$.  These distributions arise naturally
when considering equally-likely random permutations on the set of $n$ letters.
For one of the distributions, the expectation vector and covariance matrix is
derived.  For the other distribution, conjectures for several elements of the 
expectation vector are provided.

\vskip 5mm
KEY WORDS: Integer partitions, random partitions,
symmetric group, discrete probability distribution. 
\end{abstract}

\section{Background}
The study of random integer partitions is not new, but past work 
(see~\cite{F93};~\citet{CCH};~\cite{M02};~\cite{M05}) has focused on probability models
where partitions are chosen with equal probability.  
In contrast, here we assign probabilities to a given partition based on the probability that
a given permutation has cycle type indexed by that partition, as explained in detail
below.

\section{Introduction and Notations}
Let $n$ and $l$ denote a nonnegative integers.
A \emph{partition} $\lambda$ of size $n$ and length $l$ is an $l$-tuple 
$(\lambda_1, \lambda_2, \dots, \lambda_l)$ 
of integers where
 \[ \lambda_1 \geq \lambda_2 \geq \dots \lambda_l \geq 1 \] and
 \[ \lambda_1 + \lambda_2 + \cdots + \lambda_l = n. \]
 Each $\lambda_j$ is a \emph{part} of $\lambda$.  Notice that the unique partition of $0$ is the
\emph{empty partition} $\emptyset$, which has length $0$.
 
 The \emph{multiplicity} $m_j = m_j(\lambda) $ of part $j$ in $\lambda$ is the number of 
 times $j$
 appears as a part in $\lambda$.  Let the \emph{multiplicity vector} $\mv(\lambda)$ of $
 \lambda$ be given by
 \[ \mv(\lambda) := (m_1, m_2, \dots, m_n)'. \]  (For typesetting convenience, we indicate 
 a column vector
 as the transpose of a row vector, with transpose indicated by the prime ($'$) symbol.)

For a partition $\lambda$ of size $n$ and length $l$, define the corresponding 
\emph{partition vector}
$\bLambda := \bLambda(\lambda) = (\lambda_1, \lambda_2, \dots, \lambda_l, 
\lambda_{l+1} , \dots, \lambda_n)'$ where
$\lambda_j :=0$ for all $l<j\leq n$.  In other words, we take a partition $\lambda$ of $n$, 
pad it on the right with
$0$'s until its length is $n$, convert it to a column vector, and call this vector $\bLambda
$.

 For a full introduction to integer partitions the standard reference is~\citet{A}.  For a 
 gentler 
introduction to partitions suitable for undergraduates, see~\cite{AE}.

 Let $\Sym_n$ denote the symmetric group of degree $n$.  Each permutation in $
 \Sym_n$ has a cycle type 
corresponding to a partition of $n$.
For example, in $\Sym_3$, permutations (written here in disjoint cycle notation)
 $(1,2,3)$ and $(1,3,2)$ have cycle type $(3)$;
permutations $(1,3)(2)$, $(1,2)(3)$, and $(1)(2,3)$ have cycle type $(2,1)$; and
the identity permutation $(1)(2)(3) $ has cycle type $(1,1,1)$. 
\begin{table} \label{table1}
\caption{The seven partitions of $5$, and their corresponding vectors} \begin{center}
\begin{tabular}{|c|c|c|}
\hline
$\lambda$ & $\mv(\lambda)$ & $\bLambda(\lambda)$ \\
\hline
$(5)$ & $(0,0,0,0,1)'$ & $(5,0,0,0,0)'$ \\
$(4,1)$ & $(1,0,0,1,0)'$ & $(4,1,0,0,0)'$ \\
$(3,2)$ & $(0,1,1,0,0)'$ & $(3,2,0,0,0)'$ \\
$(3,1,1)$ & $(2,0,1,0,0)'$ & $(3,1,1,0,0)'$ \\
$(2,2,1)$ & $(1,2,0,0,0)'$ & $(2,2,1,0,0)'$ \\
$(2,1,1,1)$ & $(3,1,0,0,0)'$ & $(2,1,1,1,0)'$ \\
$(1,1,1,1,1)$ & $(5,0,0,0,0)'$ & $(1,1,1,1,1)'$ \\
\hline
\end{tabular} \end{center}
\end{table}

   Fix a positive integer $n$.  Select a permutation at random (each permutation is 
equally likely and
thus is chosen with probability $1/n!$).  Let the random variable $\Xv$ equal the 
partition vector $\bLambda(\lambda)$ for the
partition $\lambda$
corresponding to the cycle type of the random permutation.  Then let $\Yv := \Yv^{(n)} = 
\mv(\lambda)$.  
The support
of the distribution of $\Yv$ is those $n$-vectors $(y_1, y_2, \dots, y_n)'$ of nonnegative
integers such that 
$y_1 + 2 y_2 + 3 y_3 + \cdots n y_n = n$, i.e. those vectors that are multiplicity vectors 
for partitions
of $n$.   For any given $\Yv$, the $j$th component $Y_j$ is equal to the multiplicity 
$m_j(\lambda)$ of
$j$ in the partition $\lambda$, or, equivalently, the number of $j$-cycles in the randomly 
selected
permutation from $\Sym_n$.  (See Table 1 for the seven partitions of $5$ and their
corresponding vectors that comprise the support of $\Xv$ and $\Yv$ in the $n=5$ case.)

From the theory of the symmetric group, we know that the number of permutations in $
\Sym_n$
of cycle type $\lambda$ is \[ \frac{n!}{1^{m_1} 2^{m_2} \cdots n^{m_n} m_1! m_2! \cdots 
m_n!} ,\]
see, e.g., \citet[p. 3, Eq. (1.2)]{S}.  Thus, we may deduce that for a given permutation 
with cycle
type described by the partition $\lambda$,
 
    \[ P\big( \Xv = \bLambda(\lambda)\big) = P\big(\Yv = \mv(\lambda)\big) = \frac{1}
{1^{m_1} 2^{m_2} \cdots n^{m_n} m_1! m_2! \cdots m_n!}. \]
    
 That $\Xv$ and $\Yv$ are in fact probability distributions follows from the fact that
 \begin{equation} \label{fine} \sum \frac{1}{1^{m_1} 2^{m_2} \cdots n^{m_n} m_1! m_2! 
 \cdots m_n!} = 1, \end{equation}
where the sum is extended over all partitions $\lambda$ of $n$.
Eq.~\eqref{fine} was proved by N. J.~\citet[p. 38, Eq. (22.2)]{F}.

  In Section 2, we will study the distribution of $\Yv$, and give an explicit
formula for its expectation vector and covariance matrix. In Section 3, we find that the
expectation of $\Xv = (X_1, X_2, \dots, X_n)$ is more subtle than that of $\Yv$, and so 
we content ourselves
with some partial results and conjectures about $E(X_j)$ for various $j$.
    
\section{The distribution of the random variable $\Yv$} 
 
Before considering arbitrary $n$, let us record the pmf, expectation, and covariance
of $\Yv$ for $n=3$ and $n=5$. 
 \begin{example}[$n=3$ case] 
 In the case arising from $\Sym_3$, we have the following pmf for $\Yv$:   
\[  P(\Yv=\yv) = \left\{ \begin{array}{ll}
  1/6, & \mbox{  if $\yv = (3,0,0)'$}, \\
  1/2, & \mbox{  if $\yv=(1,1,0)'$}, \\
  1/3, & \mbox{  if $\yv=(0,0,1)'$}, \\
  0, & \mbox{   otherwise} .
\end{array}  \right.
\]
Observe that $E(\Yv) = (1, \frac 12, \frac 13)'$ and the covariance matrix
\[ \Sig =  
\threebythree{1}0{-1/3}0{1/4}{-1/6}{-1/3}{-1/6}{2/9}. \]
\end{example}

 \begin{example}[$n=5$ case]
 In the case arising from $\Sym_5$, we have the following pmf for $\Yv$:
\[  P(\Yv=\yv) = \left\{ \begin{array}{ll}
  1/120, & \mbox{  if $\yv = (5,0,0,0,0)'$}, \\
  1/12,  & \mbox{  if $\yv =(3,1,0,0,0)'$}, \\
  1/6, & \mbox{ if $\yv = (2,0,1,0,0)'$},\\
   1/8,  & \mbox{  if $\yv = (1,2,0,0,0)'$} \\
   1/4, & \mbox{  if $\yv =(1,0,0,1,0)'$}, \\
  1/6, & \mbox{  if $\yv =(0,1,1,0,0)'$}, \\
  1/5, & \mbox{ if $\yv = (0,0,0,0,1)'$}, \\
  0, & \mbox{   otherwise} .
\end{array}  \right.
\]
 Observe that $E(\Yv) = (1, \frac 12, \frac 13, \frac 14, \frac 15)'$ and the covariance 
 matrix
\[  \mathbf{\Sigma} = \CovFive . \]
 \end{example}

The preceding examples suggest that
  for general $n$, \[ E(\Yv) = \left(1, \frac 12, \frac 13, \dots, \frac 1n\right)'. \]
  Let us prove this and some other results.
  
  Let $\tv = (t_1, t_2, \dots, t_n)'$. 
 The joint moment generating function $M^{(n)}_\Yv(\tv)$ of $\Yv$ is
 \begin{align*} 
 M^{(n)}_\Yv(\tv) &= \underset{y_1 + 2y_2 + \cdots n y_n = n}{\sum_{(y_1, y_2, \dots, 
 y_n)}}
  \frac{ e^{y_1 t_1 + y_2 t_2 + \cdots + y_n t_n}} 
 {1^{y_1} 2^{y_2} \cdots n^{y_n} y_1! y_2! \cdots y_n!}
 \end{align*} 
 
 Notice that $M^{(n)}_{\Yv}(\tv) = 0$ if $n<0$, since it is the empty sum, and
 $M^{(0)}_{\Yv} = \sum_{\emptyset} \frac{e^0}{\mathrm{empty\ product}} = 1$.

In order to justify one of the steps in Theorem~\ref{thm1} below, we will require
the following simple lemma.
\begin{lemma} \label{lem0}
There exists a bijection between
the set of partitions of $n$ in which part $i$ appears at least once
and the set of unrestricted partitions of $n-i$.
\end{lemma}
\begin{proof}
  Let $\lambda$ be a partition of $n$ in which $m_i(\lambda) > 0$.  
    Map $\lambda$ to the partition 
$\mu$ for which  \[ m_j(\mu) = \left\{ \begin{array}{ll} 
  m_j(\lambda) & \mbox{ if $j\neq i$, }\\
  m_j(\lambda) - 1 & \mbox{ if $j = i$ }
  \end{array} \right. . \]
 The partition $\mu$ is clearly a partition of $n-i$, and the map is reversible and is 
 therefore a bijection.
\end{proof}

\begin{theorem} \label{thm1}
 \[ \frac{\partial}{\partial t_i} M^{(n)}_\Yv (\tv) = \frac{ e^{t_i} }{ i } M^{(n-i)}_\Yv(\tv). \]
\end{theorem}
\begin{proof}
  \begin{align*}
  \frac{\partial}{\partial t_i} M^{(n)}_\Yv (\tv) &=  
   \frac{\partial}{\partial t_i} \underset{y_1 + 2y_2 + \cdots n y_n = n}{\sum_{(y_1, y_2, \dots, y_n)}}
   \frac{ e^{y_1 t_1 + y_2 t_2 + \cdots + y_n t_n}} 
 {1^{y_1} 2^{y_2} \cdots n^{y_n} y_1! y_2! \cdots y_n!} \\
 & = \frac{\partial}{\partial t_i} 
 \underset{y_i > 0}{
 \underset{y_1 + 2y_2 + \cdots n y_n = n}{\sum_{(y_1, y_2, \dots, y_n)}}}
   \frac{ e^{y_1 t_1 + y_2 t_2 + \cdots + y_n t_n}} 
 {1^{y_1} 2^{y_2} \cdots n^{y_n} y_1! y_2! \cdots y_n!} \\
&=  \underset{y_i > 0}{
 \underset{y_1 + 2y_2 + \cdots n y_n = n}{\sum_{(y_1, y_2, \dots, y_n)}}}
   \frac{ y_i e^{y_1 t_1 + y_2 t_2 + \cdots + y_n t_n}} 
 {1^{y_1} 2^{y_2} \cdots n^{y_n} y_1! y_2! \cdots y_n!} \\
&= \frac{e^{t_i}}{i} 
{\underset{y_1 + 2y_2 + \cdots n y_n = n}{\sum_{(y_1, y_2, \dots, y_n)}}}
   \frac{i  y_i e^{y_1 t_1 + y_2 t_2 + \cdots + y_{i-1}t_{i-1} + (y_i-1)t_i + y_{i+1}t_{i+1} + 
   \cdots + y_n t_n}} 
 {1^{y_1} 2^{y_2} \cdots n^{y_n} y_1! y_2! \cdots y_n!} \\
&= \frac{e^{t_i}}{i} 
 \underset{y_1 + 2y_2 + \cdots (n-i) y_{n-i} = n-i }{\sum_{(y_1, y_2, \dots, y_{n-i} )}}
   \frac{ e^{y_1 t_1 + y_2 t_2 +  \cdots + y_{n-i} t_{n-i} }} 
 {1^{y_1} 2^{y_2} \cdots (n-i)^{y_{n-i} } y_1! y_2! \cdots y_{n-i} !} \\
&= \frac{e^{t_i}}{i} M_{\Yv}^{(n-i)} (\tv) ,
  \end{align*}
where the penultimate equality follows from Lemma~\ref{lem0}.
\end{proof}

\begin{example} For clarity,  let us observe Theorem~\ref{thm1} in the case $n=5$, 
$i=2$.
  \begin{align*} 
   \frac{\partial}{\partial t_2}    M_{\Yv}^{(5)} (\tv) &= 
     \frac{\partial}{\partial t_2}  \left(
   \frac{e^{5t_1}}{120} + \frac{e^{3t_1 + t_2}}{12} + \frac{e^{2t_1 + t_3}}{6} +
        \frac{e^{t_1 + 2 t_2}}{8} + \frac{e^{t_1 + t_4}}{4} + \frac{e^{t_2 + t_3}}{6} + 
        \frac{e^{t_5}}{5} \right) \\
        &=  1\frac{e^{3t_1 + t_2}}{12}  +
       2 \frac{e^{t_1 + 2 t_2}}{8} + 1\frac{e^{t_2 + t_3}}{6} \\
       & = \frac{e^{t_2}}{2} \left( \frac{e^{3t_1}}{6}  +
        \frac{e^{t_1 + t_2 }}{2} + \frac{e^{t_3}}{3} \right) \\
      &= \frac{e^{t_2}}{2}  M_{\Yv}^{(3)} (\tv)
  \end{align*}
 \end{example}
 
 With Theorem~\ref{thm1} in hand, we can use it to derive the moments in the usual 
 way.  
 In particular, we have the following results.
 \begin{corollary} \label{EY}
 \[
   E(\Yv) = \left( 1, 1/2, 1/3, \dots, 1/n \right)'. \]
 \end{corollary}

\begin{remark}
Corollary~\ref{EY} tells us that in $\Sym_n$, the expected number of $i$-cycles in a 
randomly selected permutation
is $1/i$ for $1\leq i \leq n$.
\end{remark}

\begin{remark}
Corollary~\ref{EY} can be deduced from~\citet[Eq. (14)]{Z}.
\end{remark}

\begin{corollary} $E(\Yv\Yv')$ can be decomposed into a sum of
 two matrices  \[ E(\Yv \Yv') = \Am + \Bm, \] where $\Am = (a_{ij})_{1\leq i, j \leq n}$ with
 \[ a_{ij} =  \left\{ \begin{array}{ll}
                         {1}/(ij), & \mbox{if $i\leq j$ and $j < n-i$, } \\
                         0 , & \mbox{if $i\leq j$ and $j\geq n-i$,} \\
                         a_{ji}, & \mbox{otherwise}
                   \end{array} \right. ;
                   \] and $\Bm = (b_{ij})_{1\leq i,j \leq n}$ with
  \[ b_{ij} =  \left\{ \begin{array}{ll}
                         1/i, & \mbox{if $i= j$,} \\          
                         0,   & \mbox{otherwise}
                   \end{array} \right. . \]
  \end{corollary}     

Of course, it is always possible to write a covariance matrix as
 \[ \mathbf{\Sigma} = E(\Yv \Yv') - E(\Yv) E(\Yv'), \] and so we may deduce the
following result.

  \begin{corollary} The entries of the covariance matrix $\Sig$ are given by
  \[  \Sigma_{ij} = \left\{  
  \begin{array}{ll}
     1/i, & \mbox{ if $i=j$ and $i\leq n/2$;} \\
     (i-1)/{i^2}, & \mbox{ if $i=j$ and $i> n/2$;} \\
    0, & \mbox{ if $i<j$ and $j\leq n-i$;} \\
    -1/(ij), & \mbox{  if $i<j$ and $j>n-i$;} \\
    \Sigma_{ji}, & \mbox{ if $i>j$. }
  \end{array}  \right.  \]
\end{corollary}
  
 \section{The distribution of the random variable $\Xv$}
 We now turn our attention to the random variable $\Xv$, first observing the pmf
in the $n=4$ case.
 \begin{example}[$n=4$ case]
In the case arising from $\Sym_4$, we have the following pmf for $\Xv$:
\[  P(\Xv=\xv) = \left\{ \begin{array}{ll}
  1/24, & \mbox{  if $\xv = (1,1,1,1)'$}, \\
  1/4,  & \mbox{  if $\xv =(2,1,1,0)'$}, \\
   1/8,  & \mbox{  if $\xv = (2,2,0,0)'$} \\
  1/3, & \mbox{  if $\xv =(3,1,0,0)'$}, \\
  1/4, & \mbox{  if $\xv =(4,0,0,0)'$}, \\
  0, & \mbox{   otherwise} .
\end{array}  \right.
\]
\end{example}                   
                   
 The moments of $\Xv$ are not as easily described as those of $\Yv$.   
 When we need to consider more than one $n$ at a time, let us write $\Xv = \Xv^{(n)} = 
 (X_1^{(n)}, X_2^{(n)}, \dots X_n^{(n)})'.$
 We note that by direct calculation, we find that $n! E(X_1^{(n)})$ for $n=1,2,3,\dots,10$ 
 is the sequence
 \[ \{1, 3, 13, 67, 411, 2911, 23\ 563, 213\ 543, 2\ 149\ 927, 23\ 759\ 791 \}, \]
 which is in the 
 \emph{Online Encyclopedia of Integer Sequences}~\citet[sequence A028418]{OEIS}.  
 We may interpret that, e.g., in $\Sym_4$, the expected largest cycle length is 
 $67/24\approx 2.792$;
 in $\Sym_5$, the expected largest cycle length is $411/120 = 3.425$, etc.      
 
 If we similarly calculate $n! E(X_2^{(n)})$ for $n=2,3,4,\dots$, we obtain the sequence
 \[ \{1, 4, 21, 131, 950, 7694, 70343, \dots \}, \] which is not directly in the OEIS, 
 although both of the
 preceding appear as subsequences of A322384 (in exactly the context at hand:
 the entry gives the ``number $T(n,k)$ of entries in the $k$th cycles of all permutations 
of $[n]$ 
when the cycles are ordered by decreasing lengths'').
 
 We can easily conclude that $E(X_n^{(n)})  = 1/n!$ since there is only one partition of 
 $n$ with $n$ parts (the partition consisting of $n$ $1$s),
 and this partition occurs with probability $1/n!$.   
 
All of the subsequent formulas are conjectural; the author guesses them based on
calculations performed with \emph{Mathematica}.  
For $n\geq 3$, 
 \begin{equation} \label{x1} 
 n!E(X^{(n)}_{n-1}) = \frac{n^2 - n + 2}{2} = \binom{n}{2} + 1;  \end{equation}
 for $n\geq 5$, 
 \begin{equation} \label{x2} 
 n! E( X^{(n)}_{n-2}) = \frac{3n^4 - 10n^3 + 21 n^2 - 14n + 24}{24} =
 3\binom{n}{4} + 2 \binom{n}{3} + \binom{n}{2} + 1
 ; \end{equation}
 and for $n\geq 7$, 
 \begin{align} n! E(X^{(n)}_{n-3}) &= 
 \frac{n^6 - 7 n^5 + 23 n^4 - 37 n^3 + 48 n^2 - 28n + 48}{48} \notag \\
 &= 15 \binom n6 + 20 \binom n5 + 9 \binom n4 + 2 \binom n3 + \binom n2 + 1. 
 \label{x3}
 \end{align}
 The idea for expressing the preceding sequence of polynomials in terms of a 
 ``binomial basis'' was suggested by the work of~\cite{BEK}.

 It appears plausible from the above and for analogous data for larger $j$, that
$n! E( X^{(n)}_{n-j})$ 
is a polynomial in $n$ of degree $2j$ for
all $n \geq 2j+1$, and that
\[ n! E( X^{(n)}_{n-j})  = \frac{n^{2j}}{ j! 2^j} + \frac{2j+1}{3\cdot 2^j (j-1)! } n^{2j-1} 
+  O(n^{2j-2}), \]
 which would follow from the conjecture that
 \begin{equation} \label{xj} n! E( X^{(n)}_{n-j})  = 1 + \sum_{i=2}^{2j} a_i \binom ni, 
 \end{equation}
for some positive integers $a_2, a_3, \dots, a_{2j}$; with $a_{2j} = (2j-1)!!$, 
$$a_{2j-1} = \frac{(2j+1)!!}{3} - (2j-1)!! ,$$ $a_4 = 9$ (when $ j\geq 3$), $a_3 = 2$, 
$a_2 = 1$.  Here we are using the double factorial notation; recall that 
\[ (2n+1)!! = \frac{(2n+1)!}{2^n n!}. \]

 \section{Conclusion}
   As already remarked, previous research on random partitions has focused on 
both unrestricted and restricted partitions of the integer $n$, but invariably with
such partitions being chosen with equal probability.  Here we consider some of the
consequences of choosing the random partition with a non-uniform, but nonetheless
natural assignment of probabilities, and observe the effect of representing the partition
in two different ways: by multiplicities and by parts.  It is hoped that this study
will help open the door to future studies of random partitions with other 
probability distributions.

 \section*{Acknowledgments}
 The author thanks Charles Champ, Dennis Eichhorn, Brandt Kronholm, Broderick 
 Oluyede, and Robert Schneider for helpful conversations.  The author thanks the
 referees for useful comments that have helped improve the exposition of this paper.
 
 \bibliographystyle{natbib-harv}

\end{document}